\documentclass[a4paper, reqno, 12pt]{amsart}
\usepackage{graphicx, enumerate, enumitem, color}
\usepackage[foot]{amsaddr}
\usepackage{amsmath,amsthm,amssymb,subfigure,xcolor}

\usepackage{hyperref}
\hypersetup{
    colorlinks,
    linkcolor={red!50!black},
    citecolor={blue!50!black},
    urlcolor={blue!80!black}
}

\usepackage[noabbrev,capitalise]{cleveref}
\crefname{lemma}{Lemma}{Lemmas}
\crefname{theorem}{Theorem}{Theorems}
\crefname{claim}{Claim}{Claims}

\usepackage[noline,ruled,titlenumbered,noend,norelsize]{algorithm2e}
\SetAlgoNoLine
\DontPrintSemicolon
\SetNlSty{footnotesize}{}{}
\IncMargin{0.5em}
\SetNlSkip{1em}
\SetKwIF{If}{ElseIf}{Else}{if}{then}{else if}{else}{endif}
\SetKw{Pick}{pick}
\SetKw{Return}{return}
\SetNlSkip{2em}
\newenvironment{algorithm-hbox}{\hbadness=10000\begin{algorithm}}{\end{algorithm}}

\newlist{exoenum}{enumerate}{3}
\setlist[exoenum,1]{label=\arabic*)}
\setlist[exoenum,2]{label=\alph*)}
\setlist[exoenum,3]{label=\roman*)}

\newcommand{\bN}{\mathbb{N}}

\renewcommand\geq{\geqslant}
\renewcommand\leq{\leqslant}

\newcommand{\be}[1]{Bad Event~#1}

\newtheorem{theorem}{Theorem}
\newtheorem{conjecture}[theorem]{Conjecture}

\newtheorem{lemma}[theorem]{Lemma}
\theoremstyle{remark}

\makeatletter
\let\old@setaddresses\@setaddresses
\def\@setaddresses{\bigskip\bgroup\parindent 0pt\let\scshape\relax\old@setaddresses\egroup}
\makeatother

\title{Progress on the adjacent vertex distinguishing edge colouring conjecture}

\begin{document}

\author[G.~Joret]{Gwena\"{e}l Joret}
\address[G.~Joret]{Computer Science Department \\
  Universit\'e Libre de Bruxelles\\
  Brussels\\
  Belgium}
\email{gjoret@ulb.ac.be}

\author[W.~Lochet]{William Lochet}
\address[W.~Lochet]{ Universit\'e C\^ote d'Azur \\ CNRS \\ I3S \\ INRIA \\ Sophia Antipolis and ENS Lyon \\ LIP \\ Lyon \\ France}
\email{william.lochet@gmail.com}

\thanks{G.\ Joret is supported by an ARC grant from the Wallonia-Brussels Federation of Belgium.}

\date{\today}

\begin{abstract} 
A proper edge colouring of a graph is {\it adjacent vertex distinguishing} if no two adjacent vertices see the same set of colours. 
Using a clever application of the Local Lemma, Hatami (2005) proved that every graph with maximum degree $\Delta$ and no isolated edge has an adjacent vertex distinguishing edge colouring with $\Delta + 300$ colours, provided $\Delta$ is large enough. 
We show that this bound can be reduced to $\Delta + 19$. 
This is motivated by the conjecture of Zhang, Liu, and Wang (2002) that $\Delta + 2$ colours are enough for $\Delta \geq 3$.   
\end{abstract}

\maketitle

\section{Introduction}

We use the notation $[n] := \{1, 2, \dots, n\}$. 
By `graph' we mean a finite, undirected, and loopless graph. 
Graphs are assumed to have no parallel edges, unless otherwise stated.   
We generally follow the terminology of Diestel~\cite{D10} for graphs, and refer the reader to this textbook for undefined terms and notations. 
Given a graph $G$, we let $N_G(u)$ denote the neighbourhood of vertex $u$ in $G$, and let $d_G(u)$ denote the degree of $u$.  
We omit the subscript $G$ when the graph is clear from the context. 

In this paper, a \textit{colouring} of a graph $G$ always means an edge colouring of $G$, defined as a mapping $c:E(G) \to \bN$ associating  integers (colours) to the edges of $G$. 
A colouring is \textit{proper} if no two adjacent vertices receive the same colour. 
Given a proper colouring $c$ of $G$ and a vertex $u\in V(G)$, we let $S_c(u)$ denote the set of colours appearing on edges incident to $u$. 
A proper colouring $c$ is {\it adjacent vertex distinguishing} (an \textit{AVD-colouring} for short) if $S_c(u) \neq S_c(v)$ for every edge $uv\in E(G)$. 
Note that this is a constraint only for edges $uv$ with $d_G(u) = d_G(v)$ since it is automatically satisfied when $d_G(u)\neq d_G(v)$. 
Zhang, Liu, and Wang~\cite{ZLW02} made the following conjecture on the numbers of colours needed in an AVD-colouring. 

\begin{conjecture}[\cite{ZLW02}]
\label{conj}
Every graph with maximum degree $\Delta\geq 3$ and no isolated edge has an AVD-colouring with $\Delta + 2$ colours.
\end{conjecture}

(The condition $\Delta\geq 3$ is required because $C_5$ needs five colours, and it is the only connected graph known to require more than $\Delta +2 $ colours.)
This conjecture captured the attention of several researchers over the years.  
It is known to be true e.g.\ if $G$ is bipartite~\cite{BGLS07}; 
if $\Delta=3$~\cite{BGLS07}; 
if $G$ is planar and $\Delta\geq 12$~\cite{BBH13,HHW14}; 
if $G$ is $2$-degenerate~\cite{WCLM16}. 
The conjecture is also known to hold asymptotically almost surely for random 4-regular graphs~\cite{GR06}. 
Using a clever application of the Local Lemma, Hatami~\cite{H05} showed that the conjecture holds up to an additive constant: He proved that every graph with maximum degree $\Delta > 10^{20}$ and no isolated edge has an AVD-colouring with at most $\Delta+300$ colours.  

In this paper, we show that the $\Delta+300$ bound of Hatami~\cite{H05} can be reduced to $\Delta+19$ (for large enough $\Delta$): 

\begin{theorem}
\label{thm:main}
There exists an integer $\Delta_0 \geq 0$ such that every graph $G$ with maximum degree $\Delta \geq \Delta_0$ and no isolated edge has an AVD-colouring with at most $\Delta+19$ colours. 
\end{theorem}

The main novelty of our approach is the use of the entropy compression method, a proof method introduced by Moser and Tard\'os~\cite{MT10} in their celebrated algorithmic proof of the Local Lemma. 
While this technique was originally developed to obtain efficient algorithms, it soon appeared that in some cases it could produce bounds that are numerically better than those following from the Local Lemma; see~\cite{DJKW16, EP13, GKM13, KM13} for early examples of such results, and~\cite{B17} for a general perspective on the method. 
Thus our main result provides yet another illustration of this phenomenon. 

One could suspect, as we do, that probabilistic methods such as the Local Lemma or entropy compression will probably not be enough on their own to prove the $\Delta+2$ bound from \cref{conj}, assuming the conjecture is true. 
Intuitively, some extra colours are bound to be `wasted' in any random process.  
Nevertheless, we find it remarkable that one can come very close to the conjectured bound using such techniques. 
If anything, we see this as extra supporting evidence for the conjecture.  

Variants and strengthenings of \cref{conj} have been proposed in the literature. 
A natural one is a list version of the problem, where each edge of $G$ has a list of colours it can use.  
Recall that for proper edge colourings, the famous list colouring conjecture states that the minimum list size such that $G$ has a proper edge colouring where each edge receives a colour from its list is equal to its chromatic index; in particular, it is at most $\Delta+1$. 
Kahn~\cite{K96} was the first to establish an upper bound of the form $\Delta + o(\Delta)$ on the required list size, and the best asymptotic bound known to this day is $\Delta + O(\sqrt{\Delta}(\log \Delta)^4)$ due to Molloy and Reed~\cite{MR00}.  
Kwa\'sny and Przyby{\l}o~\cite{KP17} recently strenghened the latter result by showing that the resulting list edge colouring can moreover be guaranteed to be an AVD-colouring.   
Hor\v{n}\'ak and Wo\'zniak~\cite{HW12} conjectured that the list edge colouring conjecture also holds in the AVD setting, and thus in particular that a $\Delta + O(1)$ bound on the list size should be enough. 

Another strengthening of \cref{conj}, due to Flandrin {\it et al.}~\cite{FMPSW13}, states that every graph with no isolated edge and maximum degree $\Delta \geq 3$ has a proper edge colouring $c$ with colours in $[\Delta + 2]$ such that, for every edge $uv$, the sum of colours of edges incident to $u$ and that of $v$ are distinct. 
This is called a {\it neighbour sum distinguishing} colouring. 
The best asympotic bound is due to Przyby{\l}o~\cite{P13,P17}, who obtained an upper bound of $\Delta + O(\Delta^{1/2})$ on the number of required colours in such a colouring. 
It is an intriguing open problem whether a $\Delta + O(1)$ bound holds; neither Hatami's proof~\cite{H05} for AVD-colourings nor ours seem to be adaptable to this setting. 
We refer the reader to~\cite{BP17,P17} and the references therein for further pointers to the relevant literature.  

The paper is organised as follows. 
In Section~\ref{sec:initial_colouring} we set up the plan for the proof of \cref{thm:main}. 
In particular, vertices with big degrees (at least $\Delta/2$, roughly) and those with small degrees are treated independently, and differently. 
Then, Section~\ref{sec:big} and Section~\ref{sec:small} are devoted to handling big and small degree vertices, respectively. 

\section{Initial colouring}
\label{sec:initial_colouring}

Fix some $\epsilon$ with $0 < \epsilon < 1/2$, which will be assumed small enough later in the proof.  
Let $G$ be a graph with maximum degree $\Delta$ and no isolated edge. 
We assume $\Delta$ is large enough for various inequalities appearing in the proof to hold.   

The beginning of our proof of \cref{thm:main} follows closely that of Hatami~\cite{H05}. 
In particular, we reuse his approach of treating differently vertices with `small' degrees and those with `big' degrees, except we use $(1/2 - \epsilon)\Delta$ as the threshold instead of $\Delta/3$ in~\cite{H05}. 
This larger threshold helps a little bit in reducing the additive constant in the main theorem; however, the bulk of the reduction from $300$ to $19$ comes from treating big degree vertices differently.  

Let $d := \lceil (1/2 - \epsilon)\Delta \rceil$. 
Taking $\Delta$ large enough, we may assume that $d < \Delta/2$. 
We begin as in~\cite{H05} by modifying the graph $G$ as follows. 
Let $G'$ be the {\em multigraph} obtained from $G$ by contracting each edge $uv\in E(G)$ such that $d_G(u) < d$ and $d_G(v) < d$ but neither $u$ nor $v$ has any other neighbour $w$ with $d_G(w) < d$. 
Then $G'$ has maximum degree $\Delta$ and maximum edge multiplicity at most $2$. 
Every proper colouring $c'$ of $G'$ can be extended to a proper colouring $c$ of $G$ with the same set of colours as follows:
For each edge $e \in E(G)$ appearing in $G'$, set $c(e) := c'(e)$. 
For each edge $uv \in E(G)$ that was contracted, we know that $d_G(u) + d_G(v) < \Delta$. 
Thus some colour $\alpha$ of $c'$ is not used on any of the edges incident to $u$ and $v$, set then $c(uv) := \alpha$. 

In~\cite{H05}, the author points out that if moreover $c'$ is an AVD-colouring of $G'$ then $c$ is an AVD-colouring of $G$. 
Using this observation, the proof in~\cite{H05} then focuses on finding an AVD-colouring of $G'$. 
This is done by starting with a proper colouring $c'$ with $\Delta+2$ colours, which exists by Vizing's theorem, and then recolouring some edges of $G'$ with new colours to obtain an AVD-colouring of $G'$. 
The advantage of working in $G'$ instead of $G$ is that the subgraph of $G'$ induced by the vertices with degree strictly less than $d$ has no isolated edge, which is important in that proof. 

Say that a vertex $u \in V(G)$ is \textit{small} if $d_{G}(u) < d$, and \textit{big} otherwise. 
Let $A$ and $B$ be the sets of small and big vertices of $G$, respectively.  
In our proof, we follow an approach similar to~\cite{H05} but we keep the focus on $G$:  
We start with a proper colouring $c'$ of $G'$ with $\Delta+2$ colours obtained from Vizing's theorem and extend it to a colouring $c$ of $G$ as in the above remark. 
Thus, the colouring $c$ uses $\Delta+2$ colours and satisfies the following property: 
\begin{equation}
\label{prop:small_edges}
S_c(u) \cap S_c(v) = \{c(uv)\} \quad \forall uv\in E(G) \textrm{ s.t.\ }  
uv \textrm{ is an isolated edge in } G[A]. 
\end{equation}
Then, we modify $c$ to obtain an AVD-colouring of $G$. 
Thus $G'$ is only used to produce the initial colouring $c$ of $G$.  
One advantage of working in $G$ is that we avoid having to deal with parallel edges, which would introduce (trivial but annoying) technicalities in our approach. 
On the other hand, a small price to pay compared to~\cite{H05} is that we will have to watch out for those edges $uv$ that are isolated in $G[A]$ in our proof.  

Our goal is to transform the colouring $c$ into an AVD-colouring of $G$. 
The plan for doing so is roughly as follows. 
First we show that we can uncolour a bounded number of edges per big vertex in such a way that edges $uv$ with $S_c(u)=S_c(v)$ and $u,v\in B$ that remain form a matching satisfying some specific properties.   
Then we show how we can recolour these uncoloured edges, plus a few other edges of $G$, to obtain a colouring where every edge $uv$ with $u,v \in B$ satisfies $S_c(u)\neq S_c(v)$. 
Finally, we recolour edges with both endpoints in $A$ in such a way that the resulting colouring is an AVD-colouring of $G$.

\section{Big vertices}
\label{sec:big}

The goal of this section is to establish the following lemma. 

\begin{lemma}
\label{lem:big}
There exists a proper colouring $c^*$ of $G$ using $\Delta + 19$ colours such that $S_{c^*}(u) \neq S_{c^*}(v)$ holds for every edge $uv \in E(G)$ with $u,v\in B$, and for every edge $uv \in E(G)$ with $u,v\in A$ which is isolated in $G[A]$. 
\end{lemma}

The resulting colouring will then be turned into an AVD-colouring of $G$ in the next section, by recolouring the non-isolated edges of $G[A]$. 
We remark that in \cref{lem:big}, we also consider edges $uv\in E(G)$ that are isolated in the graph $G[A]$, because every edge incident to $u$ or $v$ distinct from $uv$ is incident to a big vertex, and all these edges will have a fixed colour when we are done dealing with big vertices. 
Indeed, in the next section we only recolour edges in the subgraph induced by small vertices.
Thus, if $u$ and $v$ were to see the same set of colours at the end of this step, we would have no way to fix this later.

The bulk of the proof of \cref{lem:big} is done using a randomised algorithm, which we now describe. 

\subsection{A randomised algorithm}

For each vertex $u\in B$, choose an arbitrary subset $N^+(u)$ of $N(u)$ of size $d$. 
Our randomised algorithm is Algorithm~\ref{alg:big}, whose goal is to select a subset $U^+(u) \subseteq N^+(u)$ of size $2$ satisfying certain properties for each vertex $u \in B$. 
For each vertex $v\in V(G)$, we let $U^-(v):=\{u\in B: v\in U^+(u)\}$. 
Algorithm~\ref{alg:big} chooses the subsets $U^+(u)$ iteratively, one big vertex $u$ at a time. 
Hence, we see the sets $U^+(u)$ as variables, and the sets $U^-(v)$ ($v\in V(G)$) as being determined by these variables. 
(For definiteness, we set $U^+(u):=\emptyset$ for every small vertex $u$.) 
Just after choosing the subset $U^+(u)$ of a big vertex $u$, the algorithm checks whether this choice triggered any `bad event'. 
If so, the bad event is handled, which involves {\em resetting} the variable $U^+(u)$, which means setting $U^+(u):=\emptyset$, and possibly resetting other variables $U^+(v)$ for some well-chosen big vertices $v$ close to $u$ in $G$. 

Thanks to these bad events, the selected subsets satisfy a number of properties. 
A key property is that $|U^-(v)| \leq q$ for every $v\in V(G)$, with $q:= 13$ being the constant that is optimised in this proof.

At any time during the execution of the algorithm, we say that an edge $uv \in E(G)$ is \textit{selected} if $v\in U^+(u)$ or $u\in U^+(v)$. 
In the algorithm, we will make sure that if $v\in U^+(u)$ then $u\notin U^+(v)$ (that is, an edge can be selected `at most once'). 

After the algorithm terminates, selected edges will be used to fix locally the colouring $c$ obtained in Section~\ref{sec:initial_colouring} for big vertices: 
The plan is to recolour them using $q+3$ new colours, and then recolour a well-chosen matching of $G$ with yet another new colour, in such a way that at the end $S_c(u) \neq S_c(v)$ holds for all edges $uv\in E(G)$ with $u,v\in B$. 
The resulting colouring of $G$ will use $\Delta+q+6 = \Delta + 19$ colours in total. 
 
Let us explain the conventions and terminology used in Algorithm~\ref{alg:big}. 
First, we assume that the vertices of $G$ are ordered according to some fixed arbitrary ordering.  
This naturally induces an ordering of each subset of $V(G)$ as well, of each set of pairs of vertices (say using lexicographic ordering), and more generally of any set of structures built using vertices of $G$. 
This is used implicitly in what follows.  
 
We say that an edge $uv$ linking two big vertices is \textit{finished} if $U^+(w) \neq \emptyset$ for each big vertex $w$ in $N(u)\cup N(v)$ (note that this set includes $u$ and $v$). 
For $u\in V(G)$, define $S'(u)$ as the set $S_c(u)$ minus colours of edges incident to $u$ that are selected. 
That is,  
$$
S'(u):= \{c(uv): v \in N(u), v \notin U^+(u) \cup U^-(u) \}.  
$$

Call an edge $uv\in E(G)$ \textit{bad} if $u,v\in B$, $uv$ is finished, $d_G(u)=d_G(v)$, and $S'(u)=S'(v)$. 
In the algorithm, we check whether $S'(u)=S'(v)$ for two big vertices $u,v$ with $d_G(u)=d_G(v)$ only when $uv$ is finished, the idea being that if there are still big vertices adjacent to $u$ or $v$ waiting to be treated then this could potentially impact the sets $S'(u)$ and $S'(v)$.  
This motivates the notion of bad edges. 
It will not be possible to avoid bad edges entirely, however the algorithm will maintain a good control on them. 

Call an edge $uv\in E(G)$ \textit{fragile} if $uv$ is an isolated edge in $G[A]$ and $d_G(u)=d_G(v) \leq q+3$. 
The reason for introducing this notion is the following. 
Suppose $uv$ is an isolated edge in $G[A]$ with $d_G(u) = d_G(v)$.  
Observe that at the beginning $u$ and $v$ see disjoint sets of colours in the colouring $c$ except for colour $c(uv)$. 
Once the algorithm terminates, we will recolour at most $q+2$ edges incident to each big vertex $w$ (with new colours). 
Thus, if $d_G(u) = d_G(v) \geq q+4$, we know that there will be at least one edge $e$ incident to $u$ distinct from $uv$ that kept its original colour $c(e)$. 
Since $v$ does not see the colour $c(e)$, we are then assured that $u$ and $v$ see different sets of colours after the recolouring step.  
Therefore, it is only when $uv$ is fragile that we need to be careful when selecting edges incident to $u$ or $v$ to recolour. 
Fragile edges will be carefully handled in the algorithm. 

Given a big vertex $u$ with $U^+(u)=\emptyset$, an unordered pair $\{v,w\} \subseteq N^+(u)$ is \textit{admissible} for $u$ if the following three conditions are satisfied: 
\begin{itemize}
\item $v, w \notin U^-(u)$; 
\item if $vw \in E(G)$ then $vw$ is not fragile, and 
\item setting $U^+(u) := \{v,w\}$ does not create any bad edge incident to $u$. 
\end{itemize}
The algorithm considers the big vertices $u$ with $U^+(u)=\emptyset$ that remain one by one, selects the subset $U^+(u)$ randomly each time, and deals with any bad event that may occur. 
The random choice for $U^+(u)$ is made by choosing an admissible pair for $u$ uniformly at random among the first $s:={d-q \choose 2} - 3d$ admissible pairs. 
\cref{lem:invariant} below shows that there are always at least $s$ such pairs, thus this random choice can always be made.  

Five types of bad events are considered in the algorithm. 
They correspond to the five conditions tested by the {\tt if} / {\tt else if} statements; we refer to them as \be{1}, \be{2}, etc.\ in order. 
These events state the existence of certain structures in the graph. 
We remark that there could be more than one instance of the structure under consideration in the graph.  
(For instance, there could be two vertices $v\in N(u)$ with $|U^-(v)|=q+1$ in \be{1}.) 
In this case, we assume that the algorithm chooses one according to some deterministic rule. 
For the convenience of the reader, the five types of bad events considered are illustrated in Figure~\ref{fig:bad_events}. 
Let us emphasise that if any bad event is triggered, then the current vertex $u$ is always reset (i.e.\ the algorithm sets $U^+(u):=\emptyset$). 
This will ensures that no other bad event remains in the graph after dealing with the bad event under consideration. 

\begin{figure}
\centering
\includegraphics{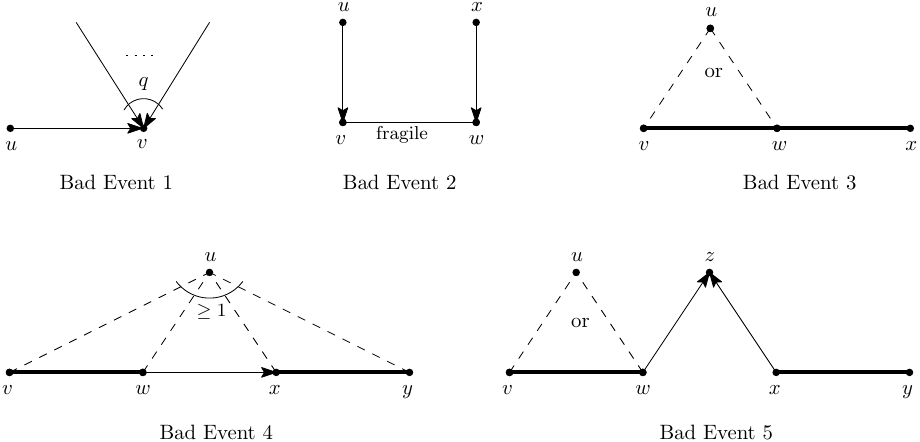}
\caption{The five types of bad events in Algorithm~\ref{alg:big}. 
Bad edges are drawn in bold. 
(Note that possibly $z=u$ in \be{5}.) 
\label{fig:bad_events}}
\end{figure}

\begin{algorithm-hbox}
\caption{Uncolouring some edges incident to big vertices.}\label{alg:big}
\SetAlgoLined 
$U^+(u) \gets \emptyset \quad \forall u\in B$ \;
\While{$\exists v\in B$ with $U^+(v)= \emptyset$}{
$u \gets$ first such vertex \; 
$U^+(u) \gets$ admissible pair chosen uniformly at random among first $s$ ones \;
\If{$\exists v \in N(u)$ with $|U^-(v)|=q+1$}{
   $U^+(w) \gets \emptyset \quad \forall w \in U^-(v)$ \;
}
\ElseIf{$\exists$ fragile edge $vw$ and $x\in V(G)-\{u,v,w\}$ such that $u \in U^-(v)$ and $x\in U^-(w)$}{
 $U^+(a) \gets \emptyset \quad \forall a \in \{u,x\}$ \;
}
\ElseIf{$\exists$ distinct bad edges $vw, wx$ with $u\in N(v) \cup N(w)$}{
 $U^+(a) \gets \emptyset \quad \forall a \in \{u,v,x\}$ \;
}  
\ElseIf{$\exists$ two independent bad edges $vw, xy$ with $u\in N(v) \cup N(w) \cup N(x) \cup N(y)$ and $x \in U^+(w)$}{
 $U^+(a) \gets \emptyset \quad \forall a \in \{u,v,w,y\}$ \;
}  
\ElseIf{$\exists$ two independent bad edges $vw, xy$ with $u\in N(v) \cup N(w)$, and $\exists z \in V(G) - \{v,w,x,y\}$ with $w, x \in U^-(z)$}{
 $U^+(a) \gets \emptyset \quad \forall a \in \{u,v,w,x,y\}$ \;
}  
}
\end{algorithm-hbox}

The rest of the proof of \cref{lem:big} consists in showing that (1) if the algorithm stops, then one can find a colouring $c^*$ satisfying the lemma, and (2) the algorithm stops with nonzero probability. 
The second step is done using the entropy compression method, a proof method introduced by Moser and Tard\'os~\cite{MT10} in their celebrated algorithmic proof of the Lov\'asz Local Lemma. 
It was pointed out by an anonymous referee that using exactly the same bad events, one could instead apply the Local Cut Lemma of Bernshteyn~\cite{B17}. This application still needs a few of the facts proved below but is overall shorter. The Local Cut Lemma itself admits a short and simple probabilistic proof that does not use the entropy compression method, so this provides a genuinely different proof. 
% \footnote{A small drawback of this approach is that the resulting proof is not algorithmic, while the proof presented in this paper is, in the following sense: A closer look at the analyses of the two randomised algorithms given in this section and the next one show that they run in expected polynomial time. As a result, the resulting AVD-colouring of $G$ can be found in expected polynomial time.}
The reader interested in following this approach will find examples of applications of the Local Cut Lemma in~\cite[Section~3]{B17}. With a little bit of work, they can be adapted to the current setting.

The following lemma establishes some key properties of Algorithm~\ref{alg:big}. 
Note that by an {\it invariant} of the {\tt while} loop, we mean a property that is true every time the condition of the loop is being tested. 
Thus, such a property holds when a new iteration of the loop starts, and also when the loop (and thus the algorithm) stops. 

%$ $   % artificial space to improve a bit the position of Lemma 3 in pdf in v1

\begin{lemma}
\label{lem:invariant}
The following properties are invariants of the {\tt while} loop in Algorithm~\ref{alg:big}: 
\begin{enumerate}
	\item $|U^-(v)| \leq q$ for every $v\in V(G)$. \label{q}
	\item At least one of $U^-(v),U^-(w)$ is empty for each fragile edge $vw$. \label{fragile}
	\item Bad edges form a matching. \label{bad_matching}
    \item If $w, x\in B$ belong to distinct bad edges, then $(\{w\} \cup U^{+}(w)) \cap (\{x\} \cup U^{+}(x)) = \emptyset$.\label{disjoint_balls}
    \item Every $u\in B$ with $U^+(u)=\emptyset$ has at least $s$ admissible pairs.  \label{admissible_pairs}
\end{enumerate}
\end{lemma}

\begin{proof}
Let us start with property~\eqref{q}. 
Clearly, $|U^-(v)| \leq q$ for every $v\in V(G)$ the first time the condition of the {\tt while} loop is being tested. 
This remains true for every subsequent test of the condition, thanks to \be{1}: 
Selecting the subset $U^+(u)$ for a vertex $u\in B$ could create up to two vertices $v$ with $|U^-(v)| = q+1$ but these get fixed immediately when $u$ is reset. 
Hence, \eqref{q} is an invariant of the loop.   

Similarly, it is clear that property \eqref{fragile} is an invariant of the {\tt while} loop, thanks to \be{2}. 

Let us consider property~\eqref{bad_matching}. 
The property is true at the beginning of the algorithm, since there are no bad edges. 
Next, suppose that property~\eqref{bad_matching} held true at the beginning of the loop but that there are two incident bad edges $e, f$ just after selecting the admissible pair $U^+(u)$ for a big vertex $u$. 
Then at least one of $e,f$, say $e$, became bad just after treating $u$. 
Note that $e$ cannot be incident to $u$, by definition of admissible pairs. 
Thus $e$ is at distance $1$ from $u$. 
It suffices to show that some bad event is triggered, since then $u$ is reset and $e$ is no longer bad (since $e$ is not finished). 
This is true, because if neither \be{1} nor \be{2} is triggered, then \be{3} is triggered because of the existence of the pair $e,f$. 
Thus we see that property~\eqref{bad_matching} is an invariant of the loop. 

The proof for property~\eqref{disjoint_balls} is similar. 
The property clearly holds at the beginning of the algorithm. 
Next, suppose that it held true at the beginning of the loop but that just after selecting the subset $U^+(u)$ for a big vertex $u$, there are two independent bad edges $vw, xy$ such that $\{w\} \cup U^{+}(w)$ and $\{x\} \cup U^{+}(x)$ intersect. 
Then at least one of the two edges, say $vw$, is at distance exactly $1$ from $u$. 
(Recall that $u$ cannot be incident to either of the two edges, by definition of admissible pairs.) 
As before, it suffices to show that some bad event occurs, since then $u$ is reset and $vw$ is no longer bad. 
Let $z\in (\{w\} \cup U^{+}(w)) \cap (\{x\} \cup U^{+}(x))$ and say none of the first four bad events happens. 
Then $z \notin \{v,w,x,y\}$, since otherwise \be{4} would have been triggered. 
But this shows that \be{5} occurs. 
We deduce that property~\eqref{disjoint_balls} is maintained. 

Finally, it remains to show that property~\eqref{admissible_pairs} is an invariant of the loop. 
Consider thus any vertex $u\in B$ with $U^+(u)=\emptyset$ when the condition of the loop is being tested.  
(Thus, a new iteration of the loop is about to start.) 
From invariant~\eqref{q}, we know that there are at least $d-q \choose 2$ unordered pairs of distinct vertices in $N^+(u) - U^-(u)$. 
Next, a key observation is that for every $x \in N(u)$, if there exists $\{v,w\} \subseteq N^+(u) - U^-(u)$ such that setting $U^+(u):=\{v,w\}$ makes the edge $ux$ bad, then the set $\{v,w\}$ is uniquely determined. 
Hence, potential bad edges forbid at most $|N(u)| \leq \Delta$ pairs of vertices in $N^+(u) - U^-(u)$. 
Finally, among the remaining pairs $\{v,w\}$, at most $\lfloor d/2 \rfloor$ of them are such that $vw$ is a fragile edge. 
Therefore, we conclude that there are at least ${d-q \choose 2} - \Delta - \lfloor d/2 \rfloor \geq {d-q \choose 2} - 3d =s$ admissible pairs for $u$. 
\end{proof}

Our interest in \cref{lem:invariant} lies in the fact that the properties listed hold in particular when Algorithm~\ref{alg:big} stops. 
While it is not clear at first sight that the algorithm should ever stop, we will show in the next subsection that it does with nonzero probability. 
In order to motivate the upcoming analysis of the algorithm, let us show that if it stops, then we can complete the proof of \cref{lem:big}. 

\begin{lemma}
\label{lem:if_stops_then}
If some execution of Algorithm~\ref{alg:big} terminates then one can find a colouring $c^*$ satisfying the requirements of \cref{lem:big}. 
\end{lemma}
\begin{proof}
Consider the resulting function $U^+$ after the algorithm stops, and the corresponding set of selected edges.  
Then $U^+$ satisfies the properties of \cref{lem:invariant}. 
Recall that each vertex of $G$ is incident to at most $q+2$ selected edges, as follows from property~\eqref{q} of \cref{lem:invariant}.  
Using Vizing's theorem we recolour the set of selected edges properly using $q+3$ new colours, say from the set $[\Delta+3, \Delta+q+5]$.  
Let $c'$ denote the resulting edge colouring of $G$. 
That is, $c'(e) := c(e)$ if $e$ was not selected, and $c'(e)$ denotes the new colour of $e$ if $e$ was selected, where $c$ is the colouring from Section~\ref{sec:initial_colouring}. 

For each bad edge $uv$, choose one of its two endpoints, say $u$, and mark one edge $uw$ for some vertex $w\in U^+(u)$, with $w\neq v$ in case $v\in U^+(u)$.  
It follows from property~\eqref{disjoint_balls} of \cref{lem:invariant} that marked edges form a matching, and that each bad edge is incident to exactly one marked edge. 
Recolouring all marked edges with a new colour, say colour $\Delta+q+6$, we obtain a proper colouring $c^*$ of $G$ with $\Delta+q+6$ colours such that $S_{c^*}(u) \neq S_{c^*}(v)$ for all edges $uv \in E(G)$ with $u,v\in B$ and $d_G(u)=d_G(v)$. 
Indeed, if $uv$ is a bad edge this holds because there is exactly one marked edge incident to $uv$, and it is distinct from $uv$ itself. 
If $uv$ is not a bad edge, then by definition $u$ and $v$ see distinct sets of colours in the colouring $c$ when considering only non-selected edges. 
Since marked edges form a subset of selected edges, we see that $S_{c^*}(u) \neq S_{c^*}(v)$ as desired.

Finally, consider edges $uv\in E(G)$ with $u, v \in A$ that are isolated in $G[A]$ with $d_G(u)=d_G(v)$.  
Recall that $d_G(u)=d_G(v) \geq 2$, since $uv$ is not isolated in $G$. 
Recall also that in the initial colouring $c$ of $G$ we had $S_c(u) \cap S_c(v) = \{c(uv)\}$, that is, $u$ and $v$ saw no common colour except for that of $uv$.   
If $uv$ is fragile, then at least one of $u,v$ is such that no incident edge was selected, by property~\eqref{fragile} of \cref{lem:invariant}, and hence $S_{c^*}(u) \neq S_{c^*}(v)$. 
If $uv$ is not fragile, then $d_G(u)=d_G(v) \geq q + 4$ by definition. 
Since at most $q+2$ edges incident to $u$ were selected, and same for $v$, we see that $u$ and $v$ are each incident to a non-selected edge distinct from $uv$. 
It follows that $S_{c^*}(u) \neq S_{c^*}(v)$. 
\end{proof}

\cref{lem:big} follows then from \cref{lem:if_stops_then} and \cref{lem:big_alg_stops} in the next subsection.  

\subsection{The algorithm terminates with high probability}

Our next result shows that Algorithm~\ref{alg:big} stops with high probability. 
For simplicity, we sometimes call one iteration of the {\tt while} loop a \textit{step}. 

\begin{lemma}
\label{lem:big_alg_stops}
The probability that Algorithm~\ref{alg:big} stops in at most $t$ steps tends to $1$ as $t\to \infty$. 
\end{lemma}

As mentioned earlier, this lemma is proved using the entropy compression method introduced by Moser and Tard\'os~\cite{MT10}.  
In a nutshell, the main idea of the proof is to look at sequences of $t$ random choices such that the algorithm does not stop in at most $t$ steps.  
Exploiting the fact that the algorithm did not stop, we show how one can get an implicit lossless encoding of these sequences, by writing down a concise log of the execution of the algorithm. 
Then, looking at the structure of the algorithm, we prove that there are only $o\left(s^t\right)$ such logs. 
Since in total there are $s^t$ random sequences of length $t$, we deduce that only a $o(1)$-fraction of these make the algorithm run for at least $t$ steps. \cref{lem:big_alg_stops} follows. 

To describe the log of an execution of the algorithm, we need the following definitions. 
First, recall that a \textit{Dyck word} of semilength $k$ is a binary word $w_1w_2\dots w_{2k}$ with exactly $k$ 0s and $k$ 1s such that the number of 0s is at least the number of 1s in every prefix of the word. 
A \textit{descent} in a Dyck word is a maximal sequence of consecutive 1s, its \textit{length} is the number of 1s. 

For our purposes, it will be more convenient to drop the requirement that a Dyck word has the same number of 0s and 1s.   
Let us define a \textit{partial Dyck word} of semilength $k$ as a binary word $w_1w_2\dots w_{p}$ with exactly $k$ 0s and {\em at most} $k$ 1s such that the number of 0s is at least the number of 1s in every prefix of the word. 
Descents are defined in the same way as for normal Dyck words. 

Let us consider a sequence $(r_1,\dots, r_t)$ of $t$ random choices such that Algorithm~\ref{alg:big} does not stop in at most $t$ steps when run with these random choices. 
In other words, the algorithm is about to start its $(t+1)$-th iteration of the {\tt while} loop, at which point we freeze its execution. 
Each random choice $r_i$ consisted in choosing an admissible pair for some big vertex $u$ among its first $s$ admissible pairs, thus we see $r_i$ as a number in $[s]$.   

For each $i\in [t+1]$, let $U^+_i$ and $U^-_i$ denote the functions $U^+$ and $U^-$, respectively, at the beginning of the $i$-th iteration, and let $B_i$ denote the subset of vertices $u\in B$ with $U^+_i(u) = \emptyset$.    
We associate to the sequence $(r_1,\dots, r_t)$ a corresponding \textit{log} $(W, \gamma, \delta, U^+_{t+1})$, where $W$ is a partial Dyck word of semilength $t$ such that the length of each descent is in the set $\{2, 3, 4, 5, q+1\}$, and $\gamma=(\gamma_1, \dots, \gamma_t)$ and $\delta=(\delta_1,\dots, \delta_t)$ are two sequences of integers. 

The partial Dyck word $W$ is built as follows during the execution of the algorithm: Starting with the empty word, we add a $0$ at the end of the word each time a big vertex is treated.   
If the corresponding random choice triggers a bad event, we moreover add $\ell$ 1s at the end of the word, where $\ell$ is the number of big vertices that are reset (so $\ell=q+1, 2, 3, 4, 5$ for bad events of types $1, 2, 3, 4, 5$, respectively). 
Thus descents in $W$ are in bijection with bad events treated during the execution, and the length of a descent tells us the type of the corresponding bad event. 

The two sequences $\gamma$ and $\delta$ are defined as follows.  
For $i\in [t]$, the integers $\gamma_i$ and $\delta_i$ encode information about the bad event handled during iteration $i$. 
If there was none, we simply set $\gamma_i:=\delta_i:=-1$.  
Otherwise, $\gamma_i$ is a nonnegative integer encoding the set of big vertices that are reset when the bad event is handled, and $\delta_i$ is a nonnegative integer encoding extra information which will help us recover the random choice $r_i$ from the log. 
The precise definitions of $\gamma_i$ and $\delta_i$ depend on the type of the bad event (see the list below); however, before giving these definitions we must explain the assumptions we make. 

The definition of $\gamma_i$ assumes that the set $B_i$ is known. 
In turn, $\gamma_i$ will encode enough information to determine completely $B_{i+1}$ from $B_i$. 
Since $B_1 = B$, it then follows that we can read off all the sets $B_1, B_2, \dots, B_{t+1}$ from the sequence $\gamma$: 
For $i=1, \dots, t$, either $\gamma_i \geq 0$, in which case $B_{i+1}$ is determined by $B_i$ and $\gamma_i$.  
Or $\gamma_i = -1$, in which case no bad event occurred during iteration $i$, and thus $B_{i+1}:=B_i-\{u\}$ where $u$ is the first vertex in $B_i$. 

As already mentioned, the purpose of the log is to encode all $t$ random choices $r_1, \dots, r_t$ that have been made during the execution. 
To encode $r_i$ ($i\in [t]$), we work {\em backwards}: 
We assume that the function $U^+_{i+1}$ is known, and we show that one can then deduce $r_i$ and $U^+_{i}$ using the log. 
Since $U^+_{t+1}$ is part of the log, this implies that the log uniquely determines $r_t, r_{t-1}, \dots, r_1$, as desired. 
Let us remark that if no bad event occurred during the $i$-th iteration, then we can already deduce $r_i$ and $U^+_{i}$ from $U^+_{i+1}$ using the sets $B_i$ and $B_{i+1}$. 
Indeed, in this case $B_i = B_{i+1} \cup \{u\}$ where $u$ is the vertex treated during the $i$-th iteration. 
Thus, for $v\in B$, 
$$
U^+_{i}(v) = \left\{ 
\begin{array}{ll}
U^+_{i+1}(v) & \textrm{ if } v \neq u \\
\emptyset & \textrm{ if } v = u
\end{array} 
\right.
$$
Furthermore, $U^+_{i+1}(u)$ tells us what was the random choice $r_i$ that was made for $u$ during iteration $i$. 
Indeed, using $U^+_{i}$ we can deduce what was the set of admissible pairs for $u$ at the beginning of iteration $i$.  
Then, $r_i$ is the position of the pair $U^+_{i+1}(u)$ in the ordering of these admissible pairs. 
Therefore, it is only when a bad event happens that we need extra information to determine $r_i$ and $U^+_{i}$.   
This is precisely the role of $\delta_i$. 

{\bf Definitions of $\gamma$ and $\delta$.}   
Let $i\in [t]$. 
If no bad event occurred during iteration $i$, set $\gamma_i:=-1$ and $\delta_i:=-1$. 
Otherwise, say that a bad event $\beta$ of type $j$ was handled.  
The definition of $\gamma_i$ assumes that $B_i$ is known, while that of $\delta_i$ assumes that $B_i$ and $U^+_{i+1}$ are both known.  
In particular, we know the vertex $u$ treated at the beginning of the iteration, since it is the first vertex in $B_i$. 
With these remarks in mind, $\gamma_i$ and $\delta_i$ are defined as follows:

\begin{itemize}
\item[$j=1$~] 
The bad event $\beta$ was triggered because the admissible pair chosen for $u$ contained a vertex $v$ with $|U^-_i(v)|=q$. 
Vertex $u$ and the $q$ vertices in $U^-_i(v)$ were subsequently reset. 
There are at most $d$ choices for vertex $v$ and at most $\binom{\Delta}{q}$ choices for $U^-_i(v)$. 
We may thus encode $v$ and $U^-_i(v)$ with a number $\gamma_i \in \left[d\binom{\Delta}{q}\right]$.  
Observe that $B_{i+1} = B_i \cup U^-_i(v)$. \\

\noindent Now that $v$ and $U^-_i(v)$ are identified, we want to encode the admissible pair $\{v,x\}$ that was chosen for $u$ at the beginning of the iteration, and the sets $U^+_i(w)$ for each vertex $w\in U^-_i(v)$. 
There are at most $d$ choices for $x$, and similarly for each $w\in U^-_i(v)$ there are at most $d$ choices for the vertex in $U^+_i(w)$ which is distinct from $v$. 
We let $\delta_i \in \left[ d^{q+1} \right]$ encode these choices. 
Since $U^+_{i}$ only differs from $U^+_{i+1}$ on vertices $w\in U^-_i(v)$, with the encoded information we can deduce $U^+_{i}$ from $U^+_{i+1}$.  
Note also that $r_i$ is determined by the admissible pair $\{v,x\}$ that was chosen for $u$. \\

\item[$j=2$~] The bad event $\beta$ was triggered because the admissible pair chosen for $u$ contained a vertex $v$ incident to a fragile edge $vw$ with $U_i^-(w) \neq \emptyset$. 
Then two vertices were reset, namely $u$ and some vertex $x$ in $U_i^-(w)$.  
There are at most $d$ choices for vertex $v$. 
Once $v$ is identified, we know vertex $w$ since fragile edges form a matching. 
Finally, there are at most $q+2$ choices for $x$, since $d_G(w) \leq q+3$ and $x\neq v$.  
We let $\gamma_i \in \left[(q+2)d\right]$ encode $v$, $w$, and $x$.  
Observe that $B_{i+1} = B_i \cup \{x\}$. \\

\noindent 
Next, to encode $r_i$ we only need to specify the vertex in the admissible pair chosen for $u$ that is distinct from $v$ ($d$ choices). 
Similarly, there are at most $d$ possibilities for the set $U^+_i(x)$ since we know that it includes $w$. 
Hence, we can encode this information with a number $\delta_i \in \left[ d^2 \right]$. 
Note that, knowing $x$ and $U^+_i(x)$, we can directly infer $U^+_i$ from $U^+_{i+1}$, since $U^+_i(y)=U^+_{i+1}(y)$ for all $y\in B-\{x\}$. \\

\item[$j=3$~] 
After selecting the admissible pair for $u$, we had $S'(v)=S'(w)=S'(x)$ for three distinct vertices $v,w,x \in B-\{u\}$ with $vw, wx\in E(G)$ and $u\in N(v) \cup N(w)$. 
Then $u,v,x$ were reset.  
There are at most $2\Delta^3$ choices for the triple $v,w,x$ (the factor $2$ is due to the fact that $u$ can be adjacent to $v$ or $w$).  
We let $\gamma_i \in \left[ 2\Delta^3 \right]$ encode $v,w,x$. 
Observe that $B_{i+1}=B_i \cup \{v,x\}$. \\

\noindent 
Knowing $v,w,x$ and $U^+_{i+1}$, our next aim is to encode $U^+_i$ and $r_i$ using $\delta_i$. 
First, we simply encode the admissible pair $\{u_1, u_2\}$ that was chosen for $u$ during the $i$-th iteration explicitly, thus there are ${d \choose 2}$ possibilities.\footnote{A reader familiar with these types of encoding arguments might wonder why we bother resetting $u$ in the algorithm if we end up writing down $\{u_1, u_2\}$ explicitly in the encoding. 
It is indeed true that in this case we only `win' something thanks to the implicit encoding of the choices made for $v$ and $x$. 
Nevertheless, we still need to reset $u$ as well, to keep the property that the current vertex is always reset whenever a bad event happens.} 
Next, we observe that $U^+_{i}(y) = U^+_{i+1}(y)$ for every $y \in B - \{u,v,x\}$, and $U^+_{i}(u) = \emptyset$. 
Thus it only remains to encode $U^+_i(v)$ and $U^+_i(x)$.  
Here the idea is that, since at this point we know the set $U^+_{i}(w)$ and the admissible pair $\{u_1, u_2\}$ chosen for $u$, there are only $O(1)$ possibilities for the sets $U^+_i(v)$ and $U^+_i(x)$ in order to have that $S'(v)=S'(w)=S'(x)$ just before $u,v,x$ were reset. \\

\noindent 
Let us focus on the set $U^+_i(v)$, the argument for $U^+_i(x)$ will be symmetric. 
First, let us write down the following local information: 
(1) Is $w \in U^+_i(v)$? 
(2) Is $w \in U^+_i(x)$?  
(3) Is $v \in U^+_i(x)$? 
Thus there are 8 possibilities. 
(1)--(2) gives enough information to reconstruct the set $S'(w)$ just before the resets, since we already know $U^+_i(w)$ and whether $w \in \{u_1,u_2\}$ or not. 
From (3) we also know the set $S''(v) := S'(v) \cup \{c(vv'): v'\in U^+_i(v)\}$ just before the resets, since we know whether $v \in \{u_1,u_2\}$ or not, and whether $v\in U^+_i(z)$ or not for every $z\in N(v) - \{u\}$. 
Now, it only remains to observe that $U^+_i(v)$ is determined by the two sets $S''(v)$ and $S'(w)$, namely $U^+_i(v) = \{v'\in N(v): c(vv') \in S''(v) - S'(w)\}$. \\

\noindent 
Proceeding similary for the set $U^+_i(x)$ (8 possibilities again), this fully determines $U^+_i$. 
Now, given $U^+_i$ we know exactly the set of admissible pairs for $u$ at the beginning of the $i$-th iteration. 
Since we know that the pair $\{u_1, u_2\}$ was chosen, we can deduce the value of $r_i$. 
Hence, this shows that $U^+_i$ and $r_i$ can be encoded using a number $\delta_i \in [64 {d \choose 2}]$. 
(The constant $64$ could be reduced with a more careful analysis but this would not make a difference later on.) \\

\item[$j=4$~]
Here we let $\gamma_i \in \left[ 4\Delta^4 \right]$ encode the four vertices $v,w,x,y$ as seen from $u$ (the factor $4$ comes from the fact that $u$ is adjacent to at least one of them but we do not know which one). 
Since $u,v,w,y$ are reset during this iteration, we have $B_{i+1} = B_i \cup \{v,w,y\}$. \\

\noindent  
Next, we set up $\delta_i$ to encode $U^+_{i}$ and $r_i$ knowing $U^+_{i+1}$. 
As in the previous case, we encode the admissible pair $\{u_1, u_2\}$ that was chosen for $u$ during the $i$-th iteration explicitly ($d \choose 2$ choices). 
Once we know $U^+_{i}$, we know which are the admissible pairs for $u$ at the beginning of the $i$-th iteration, and thus we can determine $r_i$, exactly as before. 
Thus, it only remains to encode $U^+_{i}(v), U^+_{i}(w)$, and $U^+_{i}(y)$. \\

\noindent 
Let us start with $U^+_{i}(w)$. 
We already know that $x\in U^+_{i}(w)$, and we encode the other vertex in $U^+_{i}(w)$ explicitly ($d$ choices). \\

\noindent 
Next, consider $U^+_{i}(v)$. 
Here, the idea is the same as for \be{3}, namely once $U^+_{i}(w)$ is known there are only $O(1)$ possibilities for $U^+_{i}(v)$ to have that $S'(v)=S'(w)$ just before the resets. 
To be precise, we write down the following local information: 
(1) Is $w \in U^+_i(v)$? 
(2) Is $w \in U^+_i(x)$?  
(3) Is $w \in U^+_i(y)$?  
(4) Is $v \in U^+_i(x)$? 
(5) Is $v \in U^+_i(y)$? 
Thus there are 32 possibilities. 
(1)--(3) gives us enough information to reconstruct the set $S'(w)$ just before the resets, since we already know $U^+_i(w)$ and whether $w \in \{u_1,u_2\}$ or not. 
Similarly, (4)--(5) allow us to determine the set $S''(v) := S'(v) \cup \{c(vv'): v'\in U^+_i(v)\}$ just before the resets, which in turn determines $U^+_i(v)$ since $U^+_i(v) = \{v'\in N(v): c(vv') \in S''(v) - S'(w)\}$. \\

\noindent 
For $U^+_{i}(y)$, we proceed exactly as for $U^+_{i}(v)$, exchanging $v$ with $y$ and $w$ with $x$. 
The only difference here is that $x$ is not reset, thus $U^+_{i}(x)=U^+_{i+1}(x)$.  
We similarly conclude that there are at most $32$ possibilities for the set $U^+_i(y)$.  
In summary, we may encode all the necessary information with a number $\delta_i \in \left[2^{10} d {d \choose 2} \right]$. \\

\item[$j=5$:~] 
We let $\gamma_i \in \left[ 2\Delta^5 \right]$ encode the vertices $v,w,x,y,z$. 
(Recall that possibly $z=u$.) 
Since $u,v,w,x,y$ are reset during this iteration, we have $B_{i+1} = B_i \cup \{v,w,x,y\}$. \\

\noindent  
Next, we encode $U^+_{i}$ and $r_i$  based on $U^+_{i+1}$. 
Again, we encode  the admissible pair $\{u_1, u_2\}$ chosen for $u$ explicitly ($d \choose 2$ choices), which will determine $r_i$ once we know $U^+_{i}$. 
It only remains to encode $U^+_{i}(v), U^+_{i}(w), U^+_{i}(x)$, $U^+_{i}(y)$. \\

\noindent 
Similarly as for \be{4}, there are most $d$ possibilities for the set $U^+_{i}(w)$, since we already know that $z\in U^+_{i}(w)$. 
The same is true $U^+_{i}(x)$. \\

\noindent 
For $U^+_{i}(v)$ we proceed exactly as in the previous case, exploiting the fact that $U^+_{i}(w)$ is already encoded: 
Writing down which sets among $U^+_i(v), U^+_i(x), U^+_i(y)$ include vertex $w$, and similarly which of $U^+_i(x), U^+_i(y)$ include $v$, is enough to determine $U^+_{i}(v)$. Thus there are 32 choices.
This is also true for $U^+_{i}(y)$ since the situation is completely symmetric (swapping $v,w$ with $y,x$, respectively).  
Hence, we can record the desired information with a number $\delta_i \in  \left[2^{10} d^2 {d \choose 2} \right]$. 
\end{itemize}

Let $\mathcal{R}_t$ denote the set of sequences $(r_1, \dots, r_t)$ with each $r_i \in \left[ s \right]$ such that Algorithm~\ref{alg:big} does not stop in at most $t$ steps when using $r_1, \dots, r_t$ for the random choices. 
Also, let $\mathcal{L}_t$ denote the set of logs defined by the algorithm on these sequences. 
The following lemma follows from the discussion above. 

\begin{lemma}
\label{lem:lossless_encoding_big}
For each $t\geq 1$, there is a bijection between the two sets $\mathcal{R}_t$ and $\mathcal{L}_t$. 
\end{lemma}

Next, we bound $|\mathcal{L}_t|$ from above when $t$ is large. 
To do so we need to count some specific Dyck words where each descent is weighted with some integer:  
Given a set $E=\{(l_1, w_1), \dots, (l_k, w_k)\}$ of couples of positive integers with all $l_j$'s distinct, we let $C_{t,E}$ be the number of Dyck words of semilength $t$ where each descent has length in the set $\{l_1, \dots, l_k\}$, and each descent of length $l_j$ is weighted with an integer in $[w_j]$. 

For our purposes, we take $E:=\{(l_1, w_1), \dots, (l_5, w_5)\}$, where $(l_j, w_j)$ is determined by the characteristics of \be{$j$}: $l_j$ is the number of vertices that are reset, and $w_j$ is an upper bound on the number of values the corresponding pair $(\gamma_i, \delta_i)$ can take in the log during the corresponding $i$-th iteration of the algorithm.  
Thus, following the discussion of bad events above, we take: 

\begin{itemize}
	\item $l_1 = q+1$ and $w_1 = \binom{\Delta}{q} d^{q+2} $
	\item $l_2 = 2$ and $w_2 = (q+2)d^3$
	\item $l_3 = 3$ and $w_3 = 2^7 \Delta^3 {d \choose 2}$
	\item $l_4 = 4$ and $w_4 = 2^{12} \Delta^4  d {d \choose 2}$
	\item $l_5 = 5$ and $w_5 = 2^{11} \Delta^4  d^2 {d \choose 2}$
\end{itemize}

In our logs we deal with {\em partial} Dyck words that are weighted as above. 
The difference between the number of $0$s and $1$s in the partial Dyck word corresponds to the number of big vertices $u\in B$ for which $U^+(u)$ is currently not set; we call this quantity its {\it defect}.  
Observe that partial weighted Dyck words of semilength $t$ and defect $k$ can be mapped injectively to weighted Dyck words of semilength $t+k$ by adding $k$ occurrences of $011$ at the end, where each of the $k$ new descents of length $2$ are weighted with, say, the number $1$. 
Since $k\leq n=|V(G)|$, we obtain the following lemma. 

\begin{lemma}
	$|\mathcal{L}_t| \leq \sum_{k=0}^n C_{t+k,E}$. 
\end{lemma}

In our setting, $n$ and $s$ are fixed while $t$ varies; thus, to prove that $|\mathcal{L}_t| \in o\left( s^t \right)$, it is enough by the above lemma to show that $C_{t,E} \in o\left( s^t \right)$. 
In order to bound $C_{t,E}$ from above, we follow~\cite{EP13} and use a bijection between Dyck words and rooted plane trees.

\begin{lemma}
	The number $C_{t,E}$ is equal to the number of weighted rooted plane trees on $t+1$ vertices, where each vertex has a number of children in $E \cup \{0\}$, and for each $i\in [5]$ each vertex with $l_i$ children is weighted with an integer in $[w_i]$ (leaves are not weighted).
\end{lemma}

The proof of this lemma is essentially that of Lemma~7 in~\cite{EP13}. 

Now we use generating functions and the analytic method described e.g.\ in~\cite[Section 1.2]{D04}. 
Let 
\[
y(x) := \sum_{t\geq 1} C_{t,E} x^t
\]
denote the generating function associated to our objects, and let 
\[
\phi(x) := 1 + \sum_{i=1}^5 w_i x^{l_i}. 
\]
Then $y(x)$ satisfies $y(x) = x \phi(y(x))$. 
As noted in~\cite[Theorem~5]{D04} (see also~\cite[p.278, Proposition~IV.5]{FS09}), the following asymptotic bound holds for $C_{t,E}$. 

\begin{theorem}\label{thm:asymptotic}
	Let $R$ denote the radius of convergence of $\phi$ and suppose that $\lim_{x \to R^-} \frac{x \phi'(x)}{\phi(x)} > 1$. 
	Then there exists a unique solution $\tau \in (0,R)$ of the equation $\tau \phi'(\tau) = \phi(\tau)$, and $C_{t,E} = O(\gamma^t) $, where $\gamma := \phi(\tau)/\tau$.
\end{theorem}

The radius of convergence of our function $\phi$ is $R=\infty$, and $\lim_{x \to \infty} \frac{x \phi'(x)}{\phi(x)} > 1$, thus the theorem applies. 
For our purposes, it is not necessary to compute exactly $\tau$, a good upper bound on $\gamma=\phi(\tau)/\tau$ will be enough. 
To obtain such an upper bound we use the following lemma. 

\begin{lemma}\label{lem:majoration_coef}
	For every $x\in(0,R)$, if $x\phi'(x)/\phi(x)< 1$ then $\phi(\tau)/\tau < \phi(x)/x$.
\end{lemma}

\begin{proof}
	As noted in~\cite[Note IV.46]{FS09} the function $x\phi'(x)/\phi(x)$ is increasing on $(0,R)$.
	Thus, $x\phi'(x)/\phi(x)< 1$ if and only if $x < \tau$.
	Consider the function $x\phi'(x)/\phi(x)$ on $(0, \tau)$. 
	Since $x \phi'(x) / \phi(x) < 1$, we have $x \phi'(x) - \phi(x) < 0$. 
	Moreover, since $\frac{\partial}{\partial x}( \frac{\phi(x)}{x} ) = \frac{x \phi'(x) - \phi(x)}{x^2}$, we see that $\frac{\phi(x)}{x}$
	is decreasing on $(0,\tau)$.   
	Hence, $\frac{\phi(x)}{x} > \frac{\phi(\tau)}{\tau}$. 
\end{proof}

Using these tools we can bound $\gamma$ from above. 

\begin{lemma}
\label{lem:counting_big} 
$\gamma < s$ when $d$ is large enough.  
\end{lemma}

\begin{proof}
We will use \cref{lem:majoration_coef}. 
Let $\epsilon_1 > 0$ be fixed (at the end of the proof $\epsilon_1$ will be taken small enough as a function of $q=13$).  
Let 
\[
x := \left(\frac{1}{q(1+ \epsilon_1)w_1 }\right)^{1/(q+1)}.
\] 
We claim that $x\phi'(x)/\phi(x)< 1$ when $d$ is large enough. 
First, let us give some intuition: 
If we counted only the subset of weighted Dyck words where each descent is of length $l_1=q+1$ and is weighted with an integer in $[w_1]$, then the corresponding function $\phi$ would be $\phi(x) = 1 + w_1 x^{q+1}$, and one would get $\tau=\left(\frac{1}{q w_1 }\right)^{1/(q+1)}$. 
As it turns out, the value of $\tau$ for our function of $\phi$ tends to that one (from below) as $d\to \infty$, hence our choice of $\left(\frac{1}{q w_1 }\right)^{1/(q+1)}$, slightly scaled down, for $x$.  

To show $x\phi'(x)/\phi(x)< 1$, we make the following observations, each of which is self evident: 
\begin{itemize}
	\item $x\phi'(x) = \sum_{i=1}^5l_iw_i x^{l_i}$
	\item $\phi(x) \geq 1 + w_1 x^{q+1}$
	\item $x = O\left(\frac{1}{d^2}\right)$
	\item $l_i w_i = O(d^{2l_i - 1})$ for each $i \in [2,5]$.
\end{itemize}

It follows that 
\[
\frac{\sum_{i = 2}^5 l_i w_i x^{l_i}}{\phi(x)}  = O\left(\frac{1}{d}\right) 
\]
and
\[
    \frac{x\phi'(x)}{\phi(x)} 
    \leq \frac{(q+1) w_1 x^{q+1}}{1 + w_1 x^{q+1} } + O\left(\frac{1}{d}\right) 
	= \frac{\frac{1}{1+\epsilon_1} \cdot \frac{q+1}{q} }{1 + \frac{1}{(1+\epsilon_1)q}} + O\left(\frac{1}{d}\right) 
	=  \frac{\frac{1}{1+\epsilon_1} \cdot \frac{q+1}{q} }{\frac{1}{1+\epsilon_1} \cdot \frac{q+1}{q} + \frac{\epsilon_1}{1+\epsilon_1}} + O\left(\frac{1}{d}\right). 
\]

Thus $x\phi'(x)/\phi(x) < 1$ when $d$ is large enough, as claimed. 
Hence, to prove the lemma it is enough to show that $\phi(x)/x < s$ for $d$ large enough, by \cref{lem:majoration_coef}.  

Observe that 
\[
\frac{\phi(x)}{x} = \frac{1}{x} + w_1x^q + O(d).
\]
Since $s=\binom{d-q}{2} - 3d=\Theta(d^2)$, to prove that $\phi(x)/x < s$ for $d$ large enough it is enough to show that $1/x + w_1x^q < (1-\delta)s$ for some fixed $\delta >0$.  
Let 
\[
c_{q, \epsilon_1}:=(q(1+\epsilon_1))^{1/(q+1)} + \left(\frac{1}{q(1+\epsilon_1)}\right)^{q/(q+1)}.
\]
Using that $\binom{a}{b} \leq \frac{a^b}{b!}$ and $d \leq \Delta/2$, we obtain the following bound: 
\[
	\frac{1}{x} + w_1 x^{q} 
	= c_{q, \epsilon_1} \left(\binom{\Delta}{q} d^{q+2}\right)^{1/(q+1)} 
	 \leq c_{q, \epsilon_1} \left(\frac{\Delta^{2q+2}}{2^{q+2}q!} \right)^{1/(q+1)}  
	 = c_{q, \epsilon_1} \left(\frac{1}{2^{q+2}q!}\right)^{1/(q+1)} \Delta^2. 
\]

Since $s = \binom{d-q}{2} - 3d$, for any fixed $\epsilon' >0$ we have 
$s \geq \frac{1 - \epsilon'}{2} d^2 \geq \frac{(1 - \epsilon')(1/2 - \epsilon)^2}{2} \Delta^2 $ if $d$ is large enough. 
Hence, to conclude the proof it suffices to show that the following inequality holds if $\epsilon$, $\epsilon'$ and $\epsilon_1$ are chosen small enough:  
\begin{align}    
  c_{q, \epsilon_1} \left(\frac{1}{2^{q+2}q!}\right)^{1/(q+1)} < \frac{(1 - \epsilon')(1/2 - \epsilon)^2}{2}.
\end{align}
%%% c_{13, 0} = 13**(1/14) + 13**(-13/14) = 1.29346
%%% second factor =((2**15)*(13!))**(-1/14) = 0.09503 
%%% product = 0.12292 < 0.125
This is true, since $c_{q, 0} \left(\frac{1}{2^{q+2}q!}\right)^{1/(q+1)} \simeq 0.12292 < 1/8$ for $q = 13$.
\end{proof}

It follows from \cref{thm:asymptotic} and the above lemma that $C_{t,E} \in o\left( s^t \right)$, and hence $|\mathcal{L}_t| \in o\left( s^t \right)$, when $d$ is large enough. 
(To avoid any confusion, let us emphasise that here the $o(\cdot)$ notation is w.r.t.\ the variable $t$, that is, we first assume that $d$ is large enough for \cref{lem:counting_big} to hold, and then when the graph is fixed we let $t$ vary.)
Since there are $s^t$ random sequences of length $t$, \cref{lem:big_alg_stops} follows from \cref{lem:lossless_encoding_big} and \cref{lem:counting_big}. This concludes the proof of \cref{lem:big_alg_stops}.

\section{Small vertices}
\label{sec:small}

In this section we prove the following result. 

\begin{lemma}
\label{lem:small} 
Let $m$ be an integer with $m \geq 2$. 
Suppose $G$ has a colouring $c^*$ with $\Delta + m$ colours such that $S_{c^*}(u) \neq S_{c^*}(v)$ for every edge $uv \in E(G)$ with $u,v\in B$, and for every edge $uv \in E(G)$ with $u,v\in A$ which is isolated in $G[A]$. 
Then $G$ has an AVD-colouring with $\Delta + m$ colours. 
\end{lemma}

Note that \cref{thm:main} follows from \cref{lem:big} and \cref{lem:small} with $m=19$. 
Thus it only remains to prove \cref{lem:small}, which we do now.   

Let $c^*$ be a colouring of $G$ as in the statement of \cref{lem:small}. 
Thus, $S_{c^*}(u) \neq S_{c^*}(v)$ for every edge $uv \in E(G)$ with $u,v\in B$, and for every edge $uv \in E(G)$ with $u,v\in A$ which is isolated in $G[A]$. 
However, we could have $S_{c^*}(u) = S_{c^*}(v)$ for some non-isolated edges $uv$ of $G[A]$.  
Let $A'$ be the subset of vertices of $A$ that are not incident to an isolated edge of $G[A]$. 
In this section we modify the colouring $c^*$ on the graph $G[A']$ only, and make sure that $S_{c^*}(u) \neq S_{c^*}(v)$ for every $uv \in E(G)$ with $u,v\in A'$. 
Since this has no effect on the sets $S_{c^*}(u)$ for $u\in B \cup (A-A')$, the resulting colouring will be an AVD-colouring of $G$.  

First, uncolour every edge of $G[A']$ and fix an arbitrary ordering of these edges. 
We colour these edges one by one using the following iterative algorithm; at all times, we let $c_{AVD}$ denote the current partial colouring of $G$.   
Consider the first uncoloured edge $uv$ in the ordering. 
Let $s := \lceil 2 \epsilon \Delta \rceil$. 
Since $(d_G(u)-1) + (d_G(v)-1) \leq 2(d - 1) \leq \Delta - 2\epsilon \Delta$,  there are at least $s + m$ available choices for the edge $uv$ in order to maintain a proper (partial) colouring.  
In case all other edges around $u$ are already coloured, we possibly remove one colour from the set of available choices as follows: 
Say that a neighbour $w$ of $u$ in $A' - \{v\}$ is {\it dangerous} for $u$ if $d_G(u)=d_G(w)$, all edges incident to $w$ are already coloured, and $S_{c_{AVD}}(w)= S_{c_{AVD}}(u) \cup \{i\}$ for some colour $i\in [s+m]$; the colour $i$ is a {\it dangerous colour} for $u$. 
Dangerous neighbours and colours for $v$ are defined similarly.  
If $u$ has exactly one dangerous neighbour, remove the corresponding dangerous colour from the set of available choices. 
Do the same for $v$. 
Thus, there are at least $s + m - 2 \geq s$ available choices remaining for the edge $uv$. 
Colour $uv$  with a colour chosen at random among the first $s$ colours available. 

As with the algorithm from the previous section we define some bad events that could happen after colouring the edge $uv$. 
Here, we only need to consider one type of bad event: 
\[
\textrm{The edge $uv$ received a colour that was dangerous for $u$ or $v$.}
\]
If such an event happens, consider a corresponding dangerous neighbour $w$, say it was dangerous for $u$. 
Let $F$ denote the set of edges incident to $u$ in $G[A']$ that are distinct from $uv$. 
Observe that $|F|\geq 2$, since otherwise we would have removed the dangerous colour for $u$ from the available choices. 
Our ordering of the edges of $G[A']$ induces an ordering of the edges in $F$; it will be convenient to see this ordering as a {\em cyclic} ordering. 
With these notations, the bad event is handled as follows:\footnote{We note that there is nothing special about the edge just after $uw$ in the cyclic ordering of $F$, this is an arbitrary choice. We could have chosen any edge of $F$ distinct from $uw$ as long as, knowing only $uv, w$ and $F$, we are be able to deduce which edge has been chosen.}  
\[
\textrm{Uncolour $uv$ and the edge just after $uw$ in the cyclic ordering of $F$.}
\]

After possibly handling one such bad event, the algorithm proceeds with the next uncoloured edge in this fashion, until every edge is coloured.

\begin{lemma}
\label{lem:if_alg_small_stops_then}
If the algorithm terminates, then the resulting colouring $c_{AVD}$ is an AVD-colouring of $G$. 
\end{lemma}

\begin{proof}
Consider an edge $uv \in E(G)$ with $d_G(u)=d_G(v)$. 
We already know that $S_{c_{AVD}}(u) \neq S_{c_{AVD}}(v)$ if $u,v\in B$ or if $uv$ is an isolated edge in $G[A]$, so may assume that $u,v\in A'$. 
Arguing by contradiction, suppose that $S_{c_{AVD}}(u) = S_{c_{AVD}}(v)$. 
Recall that $G[A']$ has no isolated edges, thus there is at least one edge incident to $u$ or $v$ which is distinct from $uv$ in $G[A']$.  
Let $e$ be the last edge coloured by the algorithm among all such edges.  Suppose w.l.o.g.\ that $e$ is incident to $v$, say $e=vw$. 
Then, just before the edge $vw$ was coloured for the last time, vertex $u$ was dangerous for $v$, with dangerous colour $c_{AVD}(vw)$. 
Hence, a bad event has been triggered after colouring $vw$. 
The bad event that was handled by the algorithm could have been the one with edge $uv$, or another one corresponding to another edge incident to $v$ or $w$. 
In any case, the edge $vw$ got uncoloured, a contradiction.  
\end{proof}

Thanks to \cref{lem:if_alg_small_stops_then}, to conclude the proof of \cref{lem:small} it only remains to show that the algorithm terminates with nonzero probability.   
This is done in the following lemma. 

\begin{lemma}
\label{lem:ends} 
The probability that the algorithm stops in at most $t$ steps tends to $1$ as $t\to \infty$. 
\end{lemma}
\begin{proof}
The proof is very similar to the corresponding proof in the previous section (but simpler). 
Let us encode the first $t$ steps (iterations) of an execution of the algorithm with a corresponding {\it log} $(W, \gamma, \delta, c_{AVD})$, where 
\begin{itemize}
	\item $W$ is a partial Dyck word of semilength $t$, obtained by adding a $0$ (a $1$) each time an edge is coloured (uncoloured, respectively); 
	\item $\gamma=(\gamma_1, \dots, \gamma_t)$;
	\item $\delta=(\delta_1, \dots, \delta_t)$; 
	\item $c_{AVD}$ is the current colouring at the end of the $t$-th iteration. 
\end{itemize}
For each $i\in [t]$, we let $\gamma_i := -1$ and $\delta_i := -1$ in case no bad event was triggered during the $i$-th iteration. 
Otherwise, if a bad event occurred, say involving a vertex $w$ that was dangerous for one of the two endpoints of the edge $uv$ coloured during the iteration, we let $\gamma_i \in [2d]$ identify vertex $w$ knowing $uv$ (recall that $u$ and $v$ have degree at most $d$). 
Observe that this identifies also the extra edge that is uncoloured (besides the edge $uv$). 

Then, we let $\delta_i \in [2]$ identify the colours of the two edges that got uncoloured, assuming we know these two edges and the colouring $c_{AVD}$ at the end of iteration $i$. 
Observe that we already know the {\em set} of colours that was used for these two edges, these are the two colours appearing around $w$ but not around the vertex ($u$ or $v$) that triggered the bad event. 
Thus it only remains to specify the mapping of these two colours to the two edges ($2$ possibilities). 

Reading $W$ and $\gamma$ from the beginning, one can deduce which subset of the edges of $G[A']$ was coloured at any time during the execution. 
Then using the colouring $c_{AVD}$ at the end of the $t$-th iteration and working backwards, we can reconstruct the colouring $c_{AVD}$ at any time during the execution using $\gamma$ and $\delta$, and deduce in particular which random choice was made for the edge under consideration during the $i$-th iteration. 
Hence, the log $(W, \gamma, \delta, c_{AVD})$ uniquely determines the $t$ random choices that were made. 

As before, we see a random choice as a number in $[s]$. 
Let $\mathcal{R}_t$ denote the set of vectors $(r_1, \dots, r_t)$ of length $t$, where each entry is a number in $[s]$. 
Let $\mathcal{L}_t$ denote the set of logs after $t$ steps resulting from executions of the algorithm that last for at least $t$ steps. 
By the discussion above, there is an injective mapping from $\mathcal{L}_t$ to $\mathcal{R}_t$. 
Since $|\mathcal{R}_t|=s^t$, to prove \cref{lem:ends} it only remains to show that $|\mathcal{L}_t|=o(s^t)$. 

Here, a rather crude counting will do. 
First, we count the partial Dyck words $W$ of semilength $t$ that can appear in our logs. 
Each such word has only descents of length $2$. 
They can be mapped to Dyck words of semilength $t$ simply by adding the missing $1$s at the end. 
Notice that each Dyck word of semilength $t$ is the image of at most two such partial Dyck words. 
(Two of our partial Dyck words have the same image iff they are the same except one ends with $0$ and the other ends with $011$.) 
Hence, the number of our partial Dyck words of semilength $t$ is at most twice the number of Dyck words of semilength $t$, and thus is at most $2\cdot 4^t$. 

Next, given a log $(W, \gamma, \delta, c_{AVD})$, the indices $i\in [t]$ such that $\gamma_i\neq -1$ and $\delta_i \neq -1$ correspond to descents of $W$. 
Thus there are at most $t/2$ such indices, and we see that the number of possible pairs $(\gamma, \delta)$ for a given $W$ is at most $(2d)^{t/2} \cdot 2^{t/2} = (4d)^{t/2} \leq (2\Delta)^{t/2}$. 

Finally, the number of partial colourings $c_{AVD}$ of $G$ is at most $|E(G)|^{\Delta+m+1}$, and is in particular independent of $t$. 

Assuming that $\Delta$ is large enough so that $s = \lceil 2 \epsilon \Delta \rceil > (32\Delta)^{1/2}$, we conclude that 
\[
|\mathcal{L}_t| \leq 2 \cdot 4^t \cdot (2\Delta)^{t/2} \cdot |E(G)|^{\Delta+m+1} = O\left((32\Delta)^{t/2}\right)=o(s^t), 
\]
as desired. 
\end{proof}

\section*{Acknowledgements}

We thank Marthe Bonamy for inspiring discussions and Jakub Przyby\l{}o for his careful reading and helpful comments on a previous version of the paper. We are much grateful to an anonymous referee for her/his insightful comments, which greatly helped us improve the paper. 

\bibliographystyle{abbrv}
\bibliography{bibliography}

\end{document}